\documentclass[12pt]{article}

\usepackage{amsmath,amssymb,amsthm}
\usepackage[utf8]{inputenc}

\usepackage{tikz}

\newcommand\NN{\mathbb{N}}

\newcommand\RR{\mathbb{R}}

\theoremstyle{plain}
\newtheorem{theorem}{Theorem}[section]

\newtheorem{corollary}{Corollary}[section]

\theoremstyle{definition}
\newtheorem{definition}{Definition}[section]
\newtheorem{remark}{Remark}[section]
\newtheorem{example}{Example}[section]

\newcommand{\CC}{\mathbb{C}} 
\newcommand{\D}{\mathbb{D}} 

\title{Mean ergodicity of multiplication  operators in weighted Dirichlet spaces \thanks{The first author is funded by the project PID2022-139449NB-I00, through grant
MCIN/AEI/10.13039/501100011033/FEDER, UE. The second author is funded through grant PID2024-160185NB-I00 by the Generaci\'on de Conocimiento programme and through grant RYC2021-034744-I by the Ram\'on y Cajal programme from Agencia Estatal de Investigaci\'on (Spanish Ministry of Science, Innovation and Universities).}}

\author{Antonio Bonilla and Daniel Seco}

\begin{document}

\maketitle

\begin{abstract}

We characterize bounded multiplication operators in weighted Dirichlet spaces that are power bounded, Ces\`{a}ro bounded and uniformly Kreiss. Moreover, we show the equivalence in such spaces between mean ergodicity and Ces\`{a}ro boundedness for multiplication operators. We perform the same study for adjoints of multiplication operators. As a particular example, we obtain a uniform mean ergodic multiplication operator in Dirichlet spaces that fails to be power bounded. 

{\bf Keywords}: Multiplication operator, power boundedness, Ces\`{a}ro boundedness, uniformly Kreiss boundedness, mean ergodicity, Carleson measure.

{\bf MSC2020 subject classification:} Primary, 47A35. Secondary, 46E15.
\end{abstract}

\section{Introduction}

Let $X$ be a Banach space, then $B(X)$  denotes the space of bounded linear  operators defined on $X$, and $X^*$ is the space of continuous linear functionals on $X$. In our setting, $X$ will always be a Hilbert space, and thus $X^*$ can be identified naturally with $X$.

Given $T\in B(X)$, we  denote its \emph{Cesàro mean} by $M_n(T)$, which is given by
$$
M_n(T)x:=\frac{1}{n+1}\sum_{k=0}^n T^kx
$$
for all $x\in X$.

We need to recall  some definitions concerning the behaviour of the sequence of  Cesàro means $(M_n(T))_{n\in \NN}$.

\begin{definition}
A linear operator $T$ on a Banach space $X$ is called:

\begin{enumerate}

\item \emph{Power bounded} (PB) if    there is a $C>0$ such that $\|T^n\| <C$  for all  $ n$.

\item \emph{Absolutely Ces\`aro bounded} (ACB) if there exists a constant $C > 0$ such that
$$
\sup_{N \in \mathbb{N}} \frac{1}{N} \sum_{j=1}^N \|T^j x\| \leq C \|x\|\;,
$$
for all $x\in X$.

\item \emph{Uniformly Kreiss bounded} (UKB) if there is a $C$ such that
\begin{equation}\label{UKB}
\|M_{n}(\lambda T)\| \le C \;\; \mbox{ for } |\lambda |=1  \mbox{ and } n=0,1,2, \dots \end{equation}

\item \emph{Uniform mean ergodic} if $M_n(T)$ converges in norm.

\item \emph{Mean ergodic} (ME) if $M_n(T)$ converges in the strong topology of $X$.

\item \emph{Ces\`{a}ro bounded} (CB) if the sequence $(M_n(T))_{n\in \NN}$ is bounded.

\end{enumerate}
\end{definition}

We will review the relations between these properties later but it is well known that for operators acting on Hilbert spaces we have 
\[(PB) \Rightarrow (ACB)  \Rightarrow (UKB)  \Rightarrow (ME)  \Rightarrow (CB). \]
Uniform mean ergodicity will mainly be used here as a sufficient condition for (ME). 

These implications do not have a reciprocal in general and it is part of our intention to study whether any of them can be reversed under the assumption that the operator is a multiplication operator acting on a \emph{weighted Dirichlet space}: From now on we assume $\alpha > -1$ and we consider the weighted Dirichlet space $D_{\alpha}$ consisting of all analytic
functions $f(z) = \sum_{n\ge 0} a_n z^n$ defined over the unit disc $\mathbb D$ with the property that
\[\|f\|_{D_\alpha}^2:= \sum_{n\ge 0} (n+1)^{1-\alpha} |a_n|^2 <\infty.\]
 
These are all Hilbert spaces. In particular, for $\alpha =0, 1, 2$,  we obtain the classical Dirichlet $D$, Hardy $H^2$ and Bergman $A^2$ spaces. See \cite{EFKMR, Garnett, HKZ} for their respective theories.

In this paper, we characterize bounded multiplication operators in $D_\alpha$ spaces that are (PB), (CB) or (UKB). When $\alpha \geq 1$, all of our conditions turn out to be equivalent to the function being bounded by 1 in modulus, and thus our main focus is on the case that $\alpha <1$. Then, our characterizations will be given in terms of the boundedness of certain integrals or sequences of Carleson measures. Moreover, we prove that for a multiplication operator acting on $D_\alpha$ it is equivalent to be (CB) or (ME). In particular, we obtain a uniform mean ergodic multiplication operator acting on Dirichlet spaces that fails to be power bounded. For the adjoints of multiplication operators, we will perform an analogous study, leading to characterizations for (PB), (CB) or (UKB) while the equivalence of (CB) and (ME) will only be proved under additional assumptions.

As a byproduct, we correct a mistake in the literature: In \cite{FK}, the authors give an incorrect characterization of (ME) multiplication operators in Besov spaces. Here we will prove that for their $p=2$ case, that is, for the  Dirichlet space, there exist mean operators that are (ME) but not (PB).

Before we proceed, it is natural to present a few well known results in similar directions. We do so in Section \ref{exapr}. Then, in Section \ref{mainsect}, we show our results and conclude, in Section \ref{further} by proposing a few open problems.

\section{Previously known relations between different ergodic properties}\label{exapr}

Within finite-dimensional Hilbert spaces, the classes (UKB) and (PB) are equal. However, already in such context, an easy example of the failure of some of the reverse implications is due to Assani. See \cite[page 938]{Assani}, \cite[p. 10]{E} and \cite[Theorem 5.4]{AS} for more details.

\begin{example}\label{ejemplo1}
{\rm Let $X$ be any of $\RR^2$ or $\CC^2$ and $T=\left(
    \begin{array}{cc}
     -1 & 2 \\
  0 & -1 \\
\end{array}
    \right) $.
It is clear that
$$
T^n= \left(
\begin{array}{cc}
 (-1)^n& (-1)^{n-1}2n  \\
0 & (-1)^n  \\
\end{array}
\right) \;
$$
and $\sup_{n\in \NN} \| M_n(T)\| <\infty $. Then $T$ is Cesàro bounded and $\frac{\|T^nx\|}{n}$ does not converge to 0 for some $x\in X$. This implies that $T$ is not mean ergodic.
}
\end{example}

From now on, we will concentrate on infinite-dimensional phenomena. The first example of a (ME) operator which is not (PB) was given by
Hille (\cite{Hille}, where, for large $n$, $\|T^n\| \sim  n^{1/4}$). This leads to natural questions about the possible speeds of growth of $\|T^n\|$. In this sense, a further 
example of an (ME) operator $T$ on  $L^1(\mathbb Z)$ with $\limsup_n \|T^n\|/n > 0$ was obtained in
\cite{Kosek}. Certainly, $\|T^n x\|/n \rightarrow 0$ as $n \rightarrow \infty$ for every $x\in L^1(\mathbb Z)$.

One could expect the spectrum of $T$ to play some role. However, one will not achieve a complete description: L\'eka \cite{Le} used a construction from \cite{TZ} to exhibit a (CB) operator $T$ acting on a complex Hilbert space $H$ such that $\|T^ n\|/n \rightarrow 0$ and $\sigma (T ) = \{1\}$, but T is not (PB) (see \cite[ Corollary 2.5]{TZ}). Let $S = -T$ . Then obviously $\|S^n\|/n \rightarrow  0$ and $\sigma (S) = \{-1\}$. Hence $I - S$ is invertible, and so  $(I - S)H = H$ is closed. By a result in \cite{Li}, if $\frac{\|S^n\|}{n} \rightarrow 0$, $S$ is uniform mean ergodic if and only if $(I-S)H$ is closed. Thus, we have that $S$ is uniform mean ergodic. However, S is not (PB) since neither is $T$ (see also \cite{LSS}).

Moreover, one may think that the structure of the space is not as relevant as the operator itself. In this regard, consider  the Volterra operator $V$ acting on $L^p[0,1]$, $1\le p\le \infty$. It was shown in \cite{MSZ} that $I-V$
is (UKB), whereas only for $p=2$ one can go further and even obtain (PB).

Let us study now an example that (ACB) $\not\Rightarrow$ (PB): denote by $e_n$, $n\in \NN$, the standard canonical basis, $e_n:=(0, \ldots, 0,1,0,\ldots) $, in $\ell^p(\NN)$ for $1\leq p<\infty $.

The following theorem yields a variety of (ACB) operators with different behavior on $\ell^p (\NN)$.

\begin{theorem}[\cite{BeBoMP}]\label{ejemplos}
Let $T$ be the unilateral weighted backward shift on $\ell ^p(\NN)$ with $1\leq p<\infty$ defined by $Te_1:=0$ and $Te_k:=w_ke_{k-1}$ for $k>1$. If
 $w_k:=\displaystyle \left( \frac{k}{k-1}\right)^{\alpha}  $ with $0<\alpha <\frac{1}{p}$,
then $T$ acting on $\ell^p(\NN)$  is (ACB) but not (PB). When $\alpha=1/p$, then $T$ is not (CB).
\end{theorem}

Further consequences can be obtained for operators on Hilbert spaces. 

\begin{corollary} [\cite{BeBoMP}]\label{unifkreiss}
There exists an operator on a Hilbert space that is (UKB) but not (ACB).
\end{corollary}

Let us discuss the properties of $M_n(T)$ and the growth of $\|T^n\|$. From the definition, one can check with elementary tools that
\begin{equation}\label{media}
\frac{T^n}{n+1} = M_n(T)-\frac{n}{n+1} M_{n-1}(T)\;.
\end{equation}
We notice that (CB) operators satisfy $\| T^n\|=O(n)$. Moreover, Theorem \ref{ejemplos} gives an example of a (UKB) operator on $\ell^1(\NN)$ such that $\| T^n\| =(n+1)^{1-\varepsilon}$ with $0<\varepsilon <1$.

\begin{theorem}[\cite{BeBoMP}]\label{kreiss}
Let $T$ be a (UKB) operator on a Hilbert space $H$. Then \[\lim_{n\to\infty}n^{-1}\|T^n\|=0.\]
\end{theorem}

In a reflexive Banach space, any (CB) operator satisfying $\lim_{n\to\infty}n^{-1}\|T^n\|=0$ is (ME). This will allow the reader to obtain the implication (UKB) $\Rightarrow$ (ME) for Hilbert spaces. The reverse of this fails on Hilbert spaces:
\begin {example} In \cite[Example 3.1]{TZ} the operator
 \[
T:= \left(
\begin{array}{lc}
 B& B-I  \\
0 & B  \\
\end{array}
\right)
\]
 is defined on the Hilbert space $l^2(\mathbb{N}) \oplus l^2(\mathbb{N})$, where  $B$ is the backward shift in  $l^2(\mathbb{N})$. This shows that $T$ is (ME) while $n^{-1}\|T^n\|\ge 2$. From Theorem \ref{kreiss}, we deduce that $T$ is not (UKB).

\end{example}

For more information about (ME) operators, see \cite{BJS} and \cite{CL}. Additional results about mean ergodicity of operators on spaces of analytic functions can be found  in \cite{bbf,bgjj, bgjj2}.

\bigskip 

Now let us turn our attention to the weighted Dirichlet spaces presented in the Introduction.

Clearly, the monomials $e_n(z)=(n+1)^{\frac{1-\alpha}{2}} z^n$ form an orthonormal basis in
$D_\alpha$. Therefore, the operator $M_z$ of multiplication by the independent
variable is a forward weighted shift of the form
$$
M_ze_n=\left(\frac{n+2}{n+1}\right)^{\frac{1-\alpha}{2}}e_{n+1}.
$$

On the other hand, for fixed $\alpha>-1$, we can consider an equivalent norm $\|\cdot \|$ defined on $D_\alpha$ by 
\begin{equation}\label{eqn101}
\|f\|^2 = |f(0)|^2 + \frac{1}{\pi}\int_{\mathbb D} |f'|^2(1-|z|^2)^{\alpha} dA.
\end{equation}
Since we will only be interested on asymptotic properties in terms of the norm, we will sometimes use this definition and directly identify $\|\cdot\|$ with $\|\cdot\|_ {D_\alpha}$.

It is clear that $M_z$ and $M^*_z$ are (PB)  in $D_\alpha$ if $\alpha \ge 1$ because their norm is 1. Another obvious statement is that these operators are not (PB) when $\alpha <1$, which can be seen from the action of $M_{z^n}$ on the constant function $1$. Observing the behaviour over this same function also yields that $M_z$ and $M_z^*$ are not (CB) for $\alpha <0$. For the remaining properties, we can use results from \cite{AS} and \cite{BeBoMP}, to obtain the following:

\begin{theorem}[\cite{AS}]

$M_z$ and $M^*_z$ are (UKB) in $D_\alpha$ if $\alpha >0$.

\end{theorem}

\begin{theorem}[\cite{BeBoMP}]
Let $M_z$ and $M^*_z$ act on $D_\alpha$.
\begin{enumerate}
    \item[(a)] $M^*_z$ is (ACB) if $\alpha >0$.

 \item[(b)] $M_z$ is not (ACB) if $\alpha <1$.

 \item[(c)] $M_z$ and $M^*_z$ are not (CB) when $\alpha =0$.
\end{enumerate}
\end{theorem}

We are now ready to confront our goals.

\section{Ergodic properties of multiplication  operators acting on weighted Dirichlet spaces}\label{mainsect}

Let $\alpha  \ge 1$. Then we already described the situation for $M_z$ and $M^*_z$ but it turns out that we can actually answer all questions for many more symbols than $z$. Denote by $\|f\|_{\infty}= \sup_{z\in\D} |f(z)|$.

\begin{theorem}\label{largealpha} Let $\alpha \ge 1$ and $\phi \in D_\alpha$, and define $M_\phi$ and $M_\varphi$ as acting on $D_\alpha$. Then the following are equivalent:
\begin{enumerate}
    
\item[(a)] $M_{\phi}$  and $M^*_{\phi}$ are (PB).

\item[(b)] $M_{\phi}$  and $M^*_{\phi}$ are (CB).

\item[(c)]  $\|\phi\|_{\infty} \le 1$.
\end{enumerate}
\end{theorem}

Notice that the equivalence of (a) and (b) also implies the equivalence of all five conditions (PB), (ACB), (UKB), (ME) and (CB) for such operators.
\begin{proof}

The implication (a)$\Rightarrow$ (b) has already been discussed and holds in a much more general situation.

To see that (b) implies (c), we can see that when the operator $M_{\phi}$ is (CB) then the spectrum $\sigma (M_{\phi})$ is contained in closure of the unit disc. Since $\sigma (M_{\phi})= \overline{\phi(\mathbb D)},$ then $\|\phi\|_{\infty} \le 1$.

Checking that (c) $\Rightarrow$ (a)  reduces to using an equivalent norm in $D_\alpha$ that is given by an $L^2$ integral norm over the disc with a standard radial weight (for $\alpha >1$) or over the unit circle (for $\alpha=1$). It is well known that with such norm the equalities $\|M^*_{\phi}\|= \|M_{\phi}\|= \|\phi\|_{\infty}$ hold. Therefore $\|M_\phi^n\| \le 1$ for all $n$.
\end{proof}

From what we have seen, it is only natural to focus on the case when $\alpha <1$. In fact, we will often need that $\alpha >-1$ so that two norms can be considered equivalent. Hence, we will always assume that $\alpha \in (-1,1)$ in order to avoid complicating our results excessively. Some of our observations may carry to the case $\alpha \leq -1$. We will start with a study of power boundedness, then Ces\`{a}ro boundedness, mean ergodicity and, finally, uniformly Kreiss boundedness. Since $M_\phi$ is only well-defined if $\phi$ is a multiplier, we will assume this systematically. 
From now on, denote by $\phi \in M_{D_\alpha}$ the fact that $M_\phi$ is a bounded operator in $D_\alpha$. We express this by saying that $\phi$ is a \emph{multiplier} (of $D_\alpha$). In that case, we automatically understand that the action of $M_\phi$ is on $D_\alpha$. The subset $M_{D_\alpha}$ of $D_\alpha$ formed by all multipliers is a Banach space when endowed with the norm given by the multiplication operator's norm. It only seems reasonable that multiplier spaces will play a relevant role in the study of multiplier operators. Whenever $\alpha \geq 1$, $M_{D_\alpha}$ coincides with the space of bounded analytic functions with the norm $\|\cdot\|_{\infty}$. Boundedness of the symbol is a necessary condition for a multiplier in any of the $D_\alpha$ spaces but not a sufficient one as soon as $\alpha <1$. The multiplier spaces for $\alpha \in [0,1)$ were understood thanks to Stegenga \cite{S} in terms of Carleson measures (see next Subsection). When we have a negative value of $\alpha$, the space $D_\alpha$ happens to be a multiplier algebra, meaning that every element is a multiplier. This is even though the multiplier norm does not coincide with the space norm. 

\subsection{Power boundedness}

In a weighted  Dirichlet space where $-1 < \alpha <1$, the situation is more complex than that of Theorem \ref{largealpha} and we need to introduce some concepts.

A positive Borel measure on the open unit disc $\mu$ is called a \emph{Carleson measure} for $D_{\alpha}$ if there is a constant $C$ such that
$$
\int_{\mathbb D} |g|^2 d\mu \le C \|g\|^2 _{D_{\alpha}}
$$
for all $g \in D_{\alpha}$. The smallest such constant for a measure $\mu$ is called the \emph{Carleson constant} of $\mu$.

For a sequence of positive Borel measures on the open unit disc $v=\{\mu_n\}_{n\in \NN}$ we could have a uniform control. We can thus call $v$ a \emph{uniformly bounded sequence of Carleson measures} (shortened to UBSCM) for $D_{\alpha}$ if there exists a $C>0$ such that
$$
\int_{\mathbb D} |g|^2 d\mu_n \le C \|g\|^2_{D_{\alpha}}
$$
for all $g \in D_{\alpha}$.

We start by characterizing bounded multiplication operators in the weighted Dirichlet spaces that are power bounded.

\begin{theorem}\label{TPB} Let  be $-1 < \alpha <1$ and $M_\phi$ and $M^*_\phi$ act on $D_\alpha$.  $M_{\phi}$  and $M^*_{\phi}$ are (PB) if and only if $\|\phi\|_{\infty}\le 1 $ and
 $\{|(\phi^n)'|^2(1-|z|^2)^{\alpha} dA\}_{n\in \NN}$ is  a UBSCM for $D_{\alpha}$.
\end{theorem}

Notice what happens for the range $\alpha \in [0,1)$, when we assume the stronger condition that $\|\phi\|_\infty<1$: In this particular case, a multiplier of $D_\alpha$ automatically defines a (PB) multiplication operator: the sequence of Carleson measures is uniformly bounded with the same constant as that of the first element of the sequence, which is a Carleson measure because of the multiplier assumption.

\begin{proof}

Since the norm of the adjoint is equal to that of $M_\phi$, it is sufficient to prove the characterization  for $M_{\phi}$.

$(\Rightarrow) $ If the operator $M_{\phi}$ is (PB) then the spectrum $\sigma (M_{\phi})$ is contained in closure of the unit disc. Since  by \cite {BS}, $\sigma (M_{\phi})= \overline{\phi(\mathbb D)}$,  then $\|\phi\|_{\infty} \le 1$.

In order to prove  that $|(\phi^n)'|^2(1-|z|^2)^{\alpha} dA$ is  a UBSCM for $D_{\alpha}$, we take an $f \in D_\alpha$ and we can use that
\[|f \left(\phi^n\right)'|^2 \leq 2\left(|\left(\phi^n f\right)'|^2 + |\phi^n f'|^2\right),\]
to see that
\begin{eqnarray*}
\frac{1}{\pi}\int_{\mathbb D} | f |^2|(\phi^n)' |^2 (1-|z|^2)^{\alpha}dA &\le & \frac{2}{\pi}\int_{\mathbb D} |(\phi^n f)' |^2 (1-|z|^2)^{\alpha}dA \\ & + &  \frac{2}{\pi}\int_{\mathbb D} |\phi^n f' |^2 (1-|z|^2)^{\alpha}dA.
\end{eqnarray*}
The first term is bounded by some constant times $\|f\|_{D_\alpha}^2$ precisely because $M_\phi$ is (PB). The second term, precisely because $|\phi| \leq 1$ at every point. 

$(\Leftarrow)$ Again, take an $f \in D_\alpha$. From the definition of $M_\phi$, and of the equivalent norm in \eqref{eqn101}, we have
\[\|M_{\phi}^n f\|^2_{D_{\alpha}} = \|\phi ^n f\|^2_{D_{\alpha}}  \approx |\phi^n(0) f(0)|^2 + \frac{1}{\pi}\int_{\mathbb D} |(\phi^n f)'|^2(1-|z|^2)^{\alpha} dA.\]
Bear in mind that $|\phi^n(0)| \leq 1$. We use the same identification as above, that is, $(\phi^n f)' = \phi^n f' + \left(\phi^n\right)'f,$ in order to bound the norm of $M_\phi^n f$ with a constant multiple of
\[| f(0)|^2  + 2\frac{1}{\pi}\int_{\mathbb D} | \phi^{n} f' |^2(1-|z|^2)^{\alpha} dA+ 2\frac{1}{\pi}\int_{\mathbb D}|f|^2 |(\phi^n)' |^2 (1-|z|^2)^{\alpha}dA.
\]

Using once more that $\|\phi\|_\infty \leq 1$, we can bound the first two terms by $2 \|f\|^2_{D_\alpha}$.
Using the UBSCM assumption, the remaining (right-most) term is bounded by a constant times the same quantity.
\end{proof}

The theory of multipliers for the Dirichlet spaces for $\alpha \in [0,1)$ is quite rich and it should not be a surprise to the experts that Carleson measures are related with the operator theoretic properties of multipliers. We can exploit this by strengthening the assumption on $\phi$, to avoid having to deal with \emph{sequences} of measures.

\begin{corollary} \label{powerbounded} Suppose  $-1 < \alpha <1$, $\phi \in M_{D_{\alpha}}$, $\|\phi\|_{\infty}\le 1 $ and $(\frac{|\phi'|}{1-|\phi|^2})^2(1-|z|^2)^{\alpha} dA$ is a Carleson measure for $D_{\alpha}$. Then $M_{\phi}$  and $M^*_{\phi}$ are (PB).
\end{corollary}

\begin{proof}

As in the proof of the previous Theorem, we split $(\phi^n f)' = \phi^n f' + \left(\phi^n\right)'f$ to bound the norm of $M_\phi^n f$ with
\[| f(0)|^2  + \frac{2}{\pi}\int_{\mathbb D} | \phi^{n} f' |^2(1-|z|^2)^{\alpha} dA+ \frac{2}{\pi}\int_{\mathbb D}|f|^2 |(\phi^n)' |^2 (1-|z|^2)^{\alpha}dA.
\]
Under our assumption that $\|\phi\|_\infty \leq 1$, the first two terms on the right-hand side are controlled by $2\|f\|^2_{D_\alpha}$. For the other term, we simply use that $(g^n)'=ng'g^{n-1}$, so that 

\[\int_{\mathbb D} |f|^2 |(\phi^n)' |^2 (1-|z|^2)^{\alpha}dA
\le  \int_{\mathbb D}|f|^2 n^2 |\phi^{2n-2}|(1-|\phi|^2)^2 \frac{|\phi'  |^2}{ (1-|\phi|^2)^2}  (1-|z|^2)^{\alpha}dA.
\]
Now we can use that $x^{n-1}(1-x^2) < \frac{2}{n}$ for $0 \le x \le 1$, applied to $|\phi|^2$, so that the right-hand side is bounded by
\[4\int_{\mathbb D}|f|^2   \frac{|\phi'  |^2}{ (1-|\phi|^2)^2}  (1-|z|^2)^{\alpha}dA
\le  C\|f\|^2_{D_{\alpha}},
\]
where the last inequality comes as a direct consequence of the Carleson measure assumption.
\end{proof}

Recall that a \emph{reproducing kernel Hilbert space} (RKHS) $H$ over $\D$ is a Hilbert space with  the property that for each $\omega \in \mathbb D$, there exists a unique function 
$k_{\omega} \in H$ such that 
$$
f(z)= \langle f, k_z\rangle.
$$

Thus, from the Cauchy-Schwarz inequality, we obtain a pointwise control of values in terms of the norm:
$$
|f(z)|\le \|f\| \|k_z\|.
$$
 
For some choice of equivalent norms, the reproducing kernels in $D_{\alpha}$ (see, \cite{CFS})  are given by

\begin{eqnarray*}
k_z^{\alpha}(w)= \left\{ 
\begin{matrix}
 \frac{1}{(1-\overline z w)^{\alpha}}& \mbox{ if } &0< \alpha \le 1, \\
\frac{1}{\overline zw }\log\frac{1}{1-\overline z w}& \mbox{ if } &\alpha =0. 
\end{matrix}
\right.
\end{eqnarray*}

The reproducing property allows to compute the norms of the kernels. Indeed, 
\begin{eqnarray*}
\|k_z^{\alpha}\|^2= \left\{
\begin{matrix}
  \frac{1}{(1-|z|^2)^{\alpha}}& \mbox{ if } &0< \alpha \le 1, \\
\frac{1}{|z|^2 }\log\frac{1}{1- |z|^2}& \mbox{ if } &\alpha =0,
\end{matrix}
 \right.
\end{eqnarray*}

and thus the pointwise estimates that we can deduce are
\begin{eqnarray}\label{kernel}
|f(z)|^2 \le \left\{
\begin{matrix}
  C\frac{1}{(1-|z|^2)^{\alpha}}\| f\|^2_{D_{\alpha}}& \mbox{ if }& 0< \alpha \le 1, \\
 C \log\frac{1}{1- |z|^2}\| f\|^2_{D_{\alpha}}& \mbox{ if }& \alpha =0. 
\end{matrix}
\right.
\end{eqnarray}
We can use these kernel properties to give a clean condition for power boundedness.

\begin{corollary}\label{logarithmic} Let $\phi \in M_{D_{\alpha}}$ with $0\le \alpha <1$ and $\|\phi\|_\infty \le 1$. 
\begin{enumerate}
    \item[(a)]  If $\alpha >0$ and 
$\int_{\mathbb D}(\frac{|\phi'|}{1-|\phi|^2})^2 dA < \infty$,  then $M_{\phi}$  and $M^*_{\phi}$ are (PB). 

 \item[(b)] If $\alpha=0$ and  $\int_{\mathbb D} \log\frac{1}{1- |z|^2} |\frac{\phi'(z)}{1- |\phi(z)|^2}|^2 dA< \infty $,  then $M_{\phi}$  and $M^*_{\phi}$ are (PB).

\end{enumerate}

\end{corollary}

\begin{proof}
Using \eqref{kernel} we obtain
\begin{eqnarray*}
 \frac{1}{\pi}\int_{\mathbb D}|f|^2   \frac{|\phi'  |^2}{ (1-|\phi|^2)^2}  (1-|z|^2)^{\alpha}dA
\le  C\|f\|^2_{D_{\alpha}}
\end{eqnarray*}

Thus $(\frac{|\phi'|}{1-|\phi|^2})^2(1-|z|^2)^{\alpha} dA$ is  a  Carleson measure for $D_{\alpha}$. Our new result follows now from Corollary \ref{powerbounded}.
\end{proof}

The integral condition in part (a) appears in previous contributions to the theory of analytic function spaces in relation with hyperbolic area. In \cite[page 5311]{BS}, the authors give an example of a univalent function $\phi$ for which $\phi(\mathbb D)$ is dense in $\mathbb D$ and such that  $\int_{\mathbb D}(\frac{|\phi'|}{1-|\phi|^2})^2 dA$ is bounded. This justifies studying the role of univalence.

\begin{corollary} Let $\phi \in M_{D_\alpha}$ for $0<\alpha <1$. If $\phi :\mathbb D \rightarrow \mathbb D$ is univalent  and  $\phi (\mathbb D) $  has finite hyperbolic area, then $M_{\phi}$  and $M^*_{\phi}$ are (PB).
\end{corollary}

\begin{proof}
 If $G:=\phi (\mathbb D) $  has finite hyperbolic area, then 
$\int_{G}\frac{1}{(1-|z|^2)^2} dA <\infty $.

Thus \[\int_{\mathbb D}\left(\frac{|\phi'|}{1-|\phi|^2}\right)^2 dA=\int_{G}\frac{1}{(1-|z|^2)^2} dA< \infty\] and applying the previous Corollary we get our result.\end{proof}

For $\beta>0$ and $G$ an open subset of $\mathbb D$, we say that $G$ contacts $\mathbb T$ with mean order (at most) $\beta$ provided that 
$$
\int_{0}^{2\pi} 1_G(re^{i\theta})d\theta = O((1-r)^{\frac{1}{\beta}})
$$
as $r\rightarrow 1^-$, where $1_G$ is the characteristic function of the set $G$.

For example, the mean order of contact for a domain contained in a Stoltz region is at most 1. A cusp type of contact corresponds to
orders less than one. Suppose that a domain $G \subset \mathbb D$ has boundary contacting the
circle only at the point 1 and furthermore, that near 1, $G$ lies between the graph of $y = (1 -x)^2$ and its reflection
in the x-axis. Then $G$ makes contact with $\mathbb T$ with mean order at most $1/2$ (see \cite {BCM}).

Based upon the ideas  of the proof of \cite[Corollary 5.2]{BCM}, we obtain a relation between order of contact and power boundedness: 
 
\begin{corollary}
Let $0<\alpha <1$. Assume $\phi \in M_{D_\alpha}$ is univalent and $\phi(\mathbb D)$ makes contact with the unit circle with mean order $\beta <1 $. Then  $M_{\phi}$  and $M^*_{\phi}$ are (PB). 
\end{corollary}
\begin{proof}
Let $G =\phi (\mathbb D)$. The fact that $G$ makes contact with the unit circle with mean order $\beta <1$ implies that there is a constant $C$ such that for $r \in (0,1)$ we have
$$
\int_{0}^{2\pi} 1_G(re^{i\theta})d\theta \le  C((1-r)^{\frac{1}{\beta}}).
$$
 Changing variables and using polar coordinates yields
$$
\int_{\mathbb D}\left(\frac{|\phi'|}{1-|\phi|^2}\right)^2 dA = \int_{G}\left(\frac{1}{1-|z|^2}\right)^2 dA = \int_0^1 \int_{0}^{2\pi} \frac{1_G(re^{i\theta})}{(1-r^2)^2}rd\theta dr.
$$
In the last expression we make use of our assumption to bound it from above by \[C \int_0^1 \frac{1}{(1-r^2)^{2-\frac{1}{\beta}}}d\theta dr,\] 
which is finite since $\beta <1$.
\end{proof}

Let $\mu$ be a probability measure supported on a compact subset $K$ of $\mathbb T$. The potential $U_{\mu}$ of $\mu$ is defined, for every $z \in \mathbb C$, by
$$
U_{\mu} (z)  =\int_K \log \frac{e}{|z - w|}d\mu(w)
$$
and the energy $I_{\mu}$ of $\mu$ is defined by
$$
I_{\mu} =\int _KU_{\mu}(z)d\mu(z). 
$$

The logarithmic capacity of a Borel set $E \subset \mathbb T$ is then given by
$$
Cap(E) = \sup_{\mu} e^{-I_{\mu}},
$$
where the supremum is over all Borel probability measures $\mu$ with compact support contained
in $E$. Hence $E$ is of logarithmic capacity $0$  if and only
if $I_{\mu} = \infty$ for all probability measures compactly carried by $E$. If $Cap(E) = 0$, then $E$ has null
Lebesgue measure. Moreover a compact set
$K$ such that $Cap(K) = 0$ is totally disconnected. The connection between logarithmic capacity and the Dirichlet space is well known and attracts a lot of attention in the book \cite{EFKMR}.

Suppose $\phi$ is a holomorphic self-map of $\mathbb D$ and $C_{\phi}$ defines a Hilbert–Schmidt operator on the Dirichlet space. 
In that case, by \cite[Theorem 2.1]{Gallardo}, the set $\{e^{it}: |\phi (e^{it})| = 1\}$  has zero logarithmic capacity. On the other hand, by \cite[Lemma 2.2]{Piazza}, $C_{\phi}$ defines a Hilbert–Schmidt operator on the subspace of Dirichlet functions that vanish at zero if and only if $\int_{\mathbb D}(\frac{|\phi'|}{1-|\phi|^2})^2 dA$ is bounded. Thus if $\|\phi \|_{\infty} \le 1$, $\int_{\mathbb D}(\frac{|\phi'|}{1-|\phi|^2})^2 dA$ is bounded and $\phi(0) =0$, then  $\{e^{it}: |\phi (e^{it})| = 1\}$  has zero logarithmic capacity. We can actually say something stronger.
\begin{corollary} Let $0<\alpha <1$. For every compact set $K \subset \mathbb T$ of logarithmic capacity $0$, there exists $\phi$ in the Dirichlet space,  continuous on $\overline{ \mathbb D}$ with $\phi(0)=0$ such that  $K = \{e^{it}: |\phi (e^{it})| = 1\}$ and $M_{\phi}$  and $M^*_{\phi}$ are power bounded in $D_{\alpha}$.
\end{corollary}
Observe that this time the space where we require $\phi$ belongs is not the same as where we show $M_\phi$ is (PB). 

\begin{proof}
 By \cite[Theorem 4.1]{Piazza}, for every compact set $K \subset \mathbb T$ of logarithmic capacity $ 0$, there exists $\phi$ in the Dirichlet space,  continuous on $\overline{ \mathbb D}$ with $\phi(0)=0$ such that  $K = \{e^{it}: \phi (e^{it}) = 1\}$ and the composition operator $C_{\phi}$ is Hilbert-Schmidt when acting on $D_0$  and thus 
 $\int_{\mathbb D}(\frac{|\phi'|}{1-|\phi|^2})^2 dA$ is bounded. We can then apply Corollary \ref{logarithmic} (a) to obtain that $M_{\phi}$  and $M^*_{\phi}$ are power bounded in   $D_{\alpha}$ for $\alpha >0$. \end{proof}

Probably a similar proof can yield a connection with a corresponding Riesz capacity (instead of the logarithmic one) in each weighted Dirichlet space, but we did not pursue this any further since we are not aware of any Riesz capacity version of Theorem 4.1 in \cite{Piazza}.

\subsection{Ces\`{a}ro boundedness}
At this point, we are going to follow a similar path than in previous Subsection and characterize bounded multiplication operators acting on weighted Dirichlet spaces that are  Ces\`{a}ro bounded. We remind that if we assume $\phi$ is a multiplier in some space $D_\alpha$, the conclusions regarding ergodic properties refer to that space $D_\alpha$. Recall as well that $M_n$ denotes Ces\`{a}ro means while $M_\phi$ denotes a multiplier (the difference being based on whether the subscript is a number or a function).

\begin{theorem}\label{TCB} Let $-1 < \alpha <1$ and $\phi \in M_{D_\alpha}$. $M_{\phi}$ and $M^*_{\phi}$ are (CB) if and only if $\phi \equiv 1$ or $\|\phi\|_{\infty}\le 1 $ and
$\{|(\frac{1-\phi^{n+1}}{(n+1)(1- \phi)})'|^2(1-|z|^2)^{\alpha} dA\}_{n \in \NN}$ is a UBSCM for $D_{\alpha}$.
\end{theorem}

\begin{proof}
We can assume that $\phi $ is not the constant 1. 

$(\Rightarrow)$ If the operator $M_{\phi}$ is (CB) then the spectrum $\sigma (M_{\phi})$ is contained in $\overline{\mathbb D}$. Since $\sigma (M_{\phi})= \overline{\phi(\mathbb D)}$ then $\|\phi\|_{\infty} \le 1$.

In order to prove  that $\{|(\frac{1-(\phi)^{n+1}}{(n+1)(1- \phi)})'|^2(1-|z|^2)^{\alpha} dA\}$ is a UBSCM for $D_{\alpha}$, we test it against a function $f \in D_\alpha$. The action of $M_n$ over $\phi$ is telescopic yielding
\[\frac{1}{\pi}\int_{\mathbb D} |f |^2 |(\frac{1-\phi^{n+1}}{(n+1)(1- \phi)})'|^2 (1-|z|^2)^{\alpha}dA =
\frac{1}{\pi}\int_{\mathbb D} |f |^2|(M_n( M_{\phi}1))' |^2 (1-|z|^2)^{\alpha}dA, 
\]
and now we can use the usual rule for the derivative of the product to bound the right-hand side above by
\[\frac{2}{\pi}\left(\int_{\mathbb D} |(M_n( M_{\phi} f))' |^2 (1-|z|^2)^{\alpha}dA 
+ \int_{\mathbb D} |(M_n( M_{\phi}1)) f' |^2 (1-|z|^2)^{\alpha}dA\right), \]
one of which is bounded by a constant times $\|f\|^2_{D_\alpha}$ precisely because $M_\phi$ is (CB), while the other, because $\|\phi\| \leq 1$. 

$(\Leftarrow)$ We test now the Ces\`{a}ro boundedness on an $f \in D_\alpha$. It is clear by now how to use the derivative of product rule to bound $\|M_n( M_{\phi})f\|^2_{D_{\alpha}}$ from above by

\[|f(0)|^2 + \frac{2}{\pi}\int_{\mathbb D} | (M_n( M_{\phi}1)) f' |^2(1-|z|^2)^{\alpha} dA
+ \frac{2}{\pi}\int_{\mathbb D} |(M_n( M_{\phi}1))' f |^2 (1-|z|^2)^{\alpha}dA.\]

The first two terms are bounded by a constant multiple of $\|f\|^2_{D_\alpha}$ because of our assumption that $\|\phi\|_{\infty} \leq 1$, while the third one is bounded by the Carleson constant of the sequence of measures times $\|f\|^2_{D_\alpha}$.
\end{proof}

We can exploit Carleson measures, and other tools in relation with Theorem \ref{TCB} in the same way we did with Theorem \ref{TPB}, obtaining a more usable testing condition in terms of one measure only.

\begin{corollary}\label{Cesaro1} Let $\phi \in M_{D_\alpha}$ with  $-1 < \alpha <1$, $\|\phi\|_{\infty}\le 1 $ and $ |\frac{\phi'(z)}{1-\phi(z)}|^2(1-|z|^2)^{\alpha} dA $ be a  Carleson measure for $D_{\alpha}$. Then $M_{\phi}$ and $M^*_{\phi}$ are (CB).
\end{corollary}
\begin{proof}
Let us test the (CB) property on an $f \in D_\alpha$. Applying the derivative of product rule as usual and using that $|\phi|$ is bounded by $1$, we can bound $\|M_n( M_{\phi})f\|^2_{D_{\alpha}}$ from above by 
\[
| f(0)|^2 + 2\frac{1}{\pi}\int_{\mathbb D} |  f' |^2(1-|z|^2)^{\alpha} dA
 +  \frac{2}{\pi}\int_{\mathbb D} |f|^2\left|\left(\frac{1-\phi^{n+1}}{(n+1)(1- \phi)}\right)'\right|^2(1-|z|^2)^{\alpha} dA.\]
The first two terms are dominated by $2 \|f\|^2_{D_\alpha}$. 
Now notice that the modulus of the derivative in the third term is controlled by $|\phi'/(1-\phi)|$. Therefore, the Carleson condition provides then the needed estimate. 
\end{proof}

As a consequence of Corollary \ref{Cesaro1}, we can show in particular that $M_z$ is Ces\`{a}ro bounded in     $D_{\alpha}$ for $\alpha>0$. Indeed, a well known inequality (see \cite[page 79]{AP}, for instance) assures that
 \begin{equation} \label{inequality}
 \int_{\mathbb D}\frac{|f(z)|^2}{|1-z|^2}(1-|z|^2)^{\alpha} dA \le C\|f\|^2_{D_{\alpha}}.
 \end{equation}
Thus $\frac{(1-|z|^2)^{\alpha}}{|1-z|^2} dA $ is a  Carleson measure for $D_{\alpha}$.

Now we can make use of the reproducing kernel pointwise estimates we mentioned before.

\begin{corollary} Suppose $\phi \in M_{D_\alpha}$ with  $-1< \alpha <1$ and $\|\phi\|_\infty\leq 1$. Suppose any of the following hold: 
\begin{enumerate}
\item[(a)] $\alpha >0$ and  $\int_{\mathbb D}|\frac{\phi'}{1-\phi}|^2 dA< \infty$.

\item[(b)] $\alpha =0$ and $\int_{\mathbb D}\log\frac{1}{1- |z|^2} |\frac{\phi'(z)}{1- \phi(z)}|^2 dA <\infty$.

\item[(c)] $ -1<\alpha <0$, and $\log(1-\phi)\in D_{\alpha}$.
\end{enumerate}
Then $M_{\phi}$ and $M^*_{\phi}$ are (CB) (in the corresponding $D_\alpha$).
\end{corollary}

Before the proof, we would like to stress that since $D_\alpha$ spaces for negative values of $\alpha$ are formed by functions that are continuous to the boundary, the condition on part (c) actually implies that $\phi$ does not take the value $1$. Thus, this condition could already imply  a stronger property, when $\alpha <0$. As we will see later, this will already imply that $M_\phi$ is (ME), and from the properties of $M_z$ we already know that our condition cannot give (UKB).

\begin{proof}
In order to prove (a) and (b), using  (\ref{kernel}),  we obtain 
\begin{eqnarray*}
\frac{1}{\pi}\int_{\mathbb D}|f(z)|^2\left|\frac{\phi'(z)}{1-\phi(z)}\right|^2(1-|z|^2)^{\alpha} dA.
 \le  C \|f\|^2_{D_{\alpha}}.
\end{eqnarray*}

Thus $ |\frac{\phi'(z)}{1-\phi(z)}|^2(1-|z|^2)^{\alpha} dA $ is a  Carleson measure for $D_{\alpha}$. That means that we can apply Corollary \ref{Cesaro1}, and obtain the result.

(c) is a consequence of the fact that $D_\alpha$ is a multiplicative algebra and that $1-\phi$ is bounded from below.
\end{proof}

We consider the following example to be relevant:
\begin{example}
For $z\in \D$ and $\theta \ne 0$ define $\phi$ by \[\phi(z) = e^{i\theta } \cdot \frac{1-z}{2}\cdot e^{-\frac{1+z}{1-z}}.\] Then $M_{\phi}$ is Ces\`{a}ro bounded when acting on the Dirichlet space. To check this, notice $\|\phi\|_{\infty}\le 1$. It remains to see that $\int_{\mathbb D} \log\frac{1}{1- |z|^2} |\frac{\phi'(z)}{1- \phi(z)}|^2 dA $ is  bounded. We can deduce from the proof of \cite[Theorem 7.3]{GGM} that $\int_{\mathbb D} \log\frac{1}{1- |z|^2} |\phi'(z)|^2 dA $ is finite. Then we can use that $|1- \phi(z)|> c >0$ and obtain our claim.
\end{example}

\subsection{Mean ergodicity}

It is well known that in a reflexive Banach space, a Ces\`{a}ro bounded operator $T$ that satisfies $\frac{\|T^nx\|}{n} \rightarrow 0$,  for all $x$,
must be mean ergodic. This allows us to show the only implication between ergodic properties that we can reverse: mean ergodicity and Ces\`{a}ro boundedness coincide in our setting.

\begin{theorem}\label{TME} Let $\alpha \in (-1,1)$ and $\phi \in M_{D_\alpha}$.  Then $M_{\phi}$ is (ME) if and only if it is (CB).
\end{theorem}

\begin{proof}
$(\Rightarrow)$ Any mean ergodic operator is Ces\`{a}ro bounded.

$(\Leftarrow)$ As mentioned above, since $M_{\phi}$ is Ces\`{a}ro bounded and the space is reflexive, it is sufficient to establish that  for all $f \in D_{\alpha}$, we have $\frac{\|M_{\phi}^nf\|}{n}\rightarrow 0$ (as $n\rightarrow \infty)$.

The (CB) property implies already that  $\|\phi\|_{\infty}\le 1 $, and if $|\phi(z)|=1$ at some $z\in \mathbb D$ then $\phi$ must be constant and $\left\|\frac{M_{\phi}^n f}{n}\right\|_{D_{\alpha}}\rightarrow 0$.

We can otherwise assume that $|\phi(z)|<1$ for all $z\in \mathbb D$. We want to study the quantity \[\left\|\frac{M_{\phi}^n f}{n}\right\|^2_{D_{\alpha}} =  \left|\frac{\phi^n(0) f(0)}{n}\right|^2 + \frac{1}{n^2\pi}\int_{\mathbb D} |(\phi^n f)'|^2(1-|z|^2)^{\alpha} dA.\]
The first term on the right-hand side decays exponentially and we do not need to worry about it any further.
To deal with the remainder, we make the derivative explicit and once more, split a square of the modulus of a sum in two pieces to bound our integral above by
\[\frac{2}{n^2\pi}\int_{\mathbb D} | \phi^{n} f' |^2(1-|z|^2)^{\alpha} dA+ \frac{2}{\pi}\int_{\mathbb D}|f|^2|\phi^{n-1} |^2 |\phi' |^2 (1-|z|^2)^{\alpha}dA.\]
The first of these integrals can be bounded by $\frac{2\|f\|^2_{D_\alpha}}{n^2}$ using that $|\phi|<1$. 

Now, under our hypotheses, we have 
that $\frac{\|M_{\phi}^n f\|}{n}\rightarrow 0$: Bear in mind Lebesgue dominated convergence Theorem. It can be applied since its integrand is vanishing (because $|\phi|<1$) and bounded above by $|f|^2| |\phi' |^2 (1-|z|^2)^{\alpha}$, which is a Carleson measure, since $\phi$ is a multiplier.
\end{proof}

The first immediate consequence of Theorem \ref{TME} comes from applying to our knowledge that the shift is (CB).

\begin{corollary}
For $\alpha >0$, $M_z$ is (ME) in  $D_\alpha$.
\end{corollary}

It is reasonable to search for an analogue of Theorem \ref{TME} for the adjoints of multiplication operators. We can get almost the same result, since the assumption on $\phi$ being a multiplier is now joined by an integral condition in the cases that $\alpha \in [0,1)$. The multiplier assumption on $\phi$ guarantees that $M^*_\phi$ is well defined.

\begin{theorem} Let  $-1 < \alpha <1$ be and  $\phi\in M_{D_\alpha}$. When $\alpha \in (0,1)$ assume moreover that   $\int_{\mathbb D} |\phi'(z)|^2 dA<\infty $ (that is, $\phi$ is in the classical Dirichlet space). When $\alpha=0$, assume  $\int_{\mathbb D} \log\frac{1}{1- |z|^2} |\phi'(z)|^2 dA<\infty $. Then $M^*_{\phi}$ is (ME) if and only if it is (CB).
\end{theorem}
\begin{proof}

Without lost of generality we can assume that  $\phi $ is not a unimodular constant. 

$(\Rightarrow)$ All mean ergodic operators are Ces\`{a}ro bounded.

$(\Leftarrow)$ Suppose now $M_{\phi}^{*}$ is Ces\`{a}ro bounded. Since the space is reflexive it is sufficient to show that  $\frac{{\|M_{\phi}^{*}}^n f\|}{n}\rightarrow 0$ as $n\rightarrow \infty$.

Since $\phi$ is not a unimodular constant, we already know that  $|\phi(z)|<1$ for all $z\in \mathbb D$.

We want to control the equivalent quantity \[\left\|\frac{M_{\phi}^n f}{n} \right\|^2_{D_{\alpha}} = \left\|\frac{\phi ^n f}{n} \right\|^2_{D_{\alpha}}  =  \left|\frac{\phi^n(0) f(0)}{n}\right|^2 + \frac{1}{n^2\pi}\int_{\mathbb D} |(\phi^n f)'|^2(1-|z|^2)^{\alpha} dA.\]
As usual, the exponential decay of the value at 0 of $\phi^n$ makes the first term irrelevant.
For the remainder, the product rule for the derivative yields an upper bound of 
\[ \frac{2}{n^2\pi}\int_{\mathbb D} | \phi^{n} f' |^2(1-|z|^2)^{\alpha} dA+ \frac{2}{\pi}\int_{\mathbb D}|f|^2|\phi^{n-1} |^2 |\phi' |^2 (1-|z|^2)^{\alpha}dA,\]
and then one can bound the first of these integrals, using that $\|\phi\|_\infty \leq 1$, by a constant times $\|f\|^2_{D_\alpha}$. The other term is bounded in each case using the integral conditions we assumed in the Theorem and using Lebesgue dominated convergence theorem. For $\alpha <0$, the integral condition needed is just that $\phi$ is a multiplier of the space.

Thus $M^*_{\phi}$ is indeed Ces\`{a}ro bounded and satisfies $\frac{\|M^{*n}_{\phi} \|}{n}\rightarrow 0$. Therefore, it is mean ergodic.
\end{proof}

We obtain a few immediate consequences, as in previous sections. First, by taking a fixed Carleson measure:

\begin{corollary}  Let  $-1< \alpha <1$, $\phi\in M_{D_\alpha}$ and $\|\phi\|_{\infty}\le 1 $. If $\alpha >0$, assume moreover that $\phi$ is also in the classical Dirichlet space. If $\alpha=0$ assume that $\int_{\mathbb D} \log\frac{1}{1- |z|^2} |\phi'(z)|^2 dA<\infty $. If  $ |\frac{\phi'(z)}{1-\phi(z)}|^2(1-|z|^2)^{\alpha} dA $ is a  Carleson measure for $D_{\alpha}$ then $M_{\phi}$ and $M^*_{\phi}$ are (ME).
\end{corollary}

Then, we split the consequences of kernel estimates in terms of the value of $\alpha$:
\begin{corollary} Let $0< \alpha <1$, $\phi \in M_{D_\alpha}$ and $\|\phi\|_{\infty}\le 1 $. If $\int_{\mathbb D}  |\frac{\phi'(z)}{1- \phi(z)}|^2 dA < \infty$, then $M_{\phi}$ and $M^*_{\phi}$ are (ME).
\end{corollary}

For the classical Dirichlet space we obtain the following:
\begin{corollary}\label{coroerg} Let $\phi \in M_{D_0}$ and $\|\phi\|_{\infty}\le 1 $. If $\int_{\mathbb D}\log\frac{1}{1- |z|^2} |\frac{\phi'(z)}{1- \phi(z)}|^2 dA <\infty$, then $M_{\phi}$ and $M^*_{\phi}$ are (ME).
\end{corollary}

For $\alpha <0$, our condition can be expressed as $\phi$ avoiding the value $1$: 
\begin{corollary} Let $-1< \alpha <0$,  $\phi\in M_{D_{\alpha}}$ and $\|\phi\|_{\infty}\le 1 $. If $\log(1-\phi)\in D_{\alpha}$, then $M_{\phi}$ and $M^*_{\phi}$ are (ME).
\end{corollary}

A particular example where the previous condition can be easily checked is  given by $\phi(z) = e^{i\theta}\frac{1-z}{2}$ with $\theta \ne 0 $, and hence, such $M_{\phi}$ is (ME) in $D_\alpha$ for $\alpha \in (-1,0)$.

To conclude our study of mean ergodicity, we give an example that illustrates how the situation in the Dirichlet space varies from that in Hardy or Bergman spaces.  

\begin{example}
Let $\phi(z) = e^{i\theta} (\frac{1-z}{2})^{k} $ for some  $\theta \ne 0$ and $k\in \mathbb N$, $k \ge 1$. Then $M_{\phi}$ and $M^*_{\phi}$ acting on $D_0$ are (ME) but not (PB).
\end{example}

\begin{proof}
It is sufficient to prove the conditions of the Corollary \ref{coroerg}. To do so, notice $\|\phi\|_{\infty}\le 1 $. Moreover, it is well known that  $\int_{\mathbb D}  |1-z|^{2\beta}\log\frac{2}{1-|z|^2} dA $ is  bounded if $\beta >-1$ \cite[ p.1112]{GGM}. If we also use that $|1-\phi(z)|>c>0$ $\forall z\in \mathbb D$, we have  \[\int_{\mathbb D}\log \frac{1}{1-|z|^2} \left|\frac{\phi'(z)}{1- \phi(z)}\right|^2 dA <\infty.\]
On the other hand, since $(\frac{1-z}{2})^{m}=2^{-m}\displaystyle \sum _{k=1}^m \binom {m}{k}(-z)^k$, we can estimate the norm of $\phi$ in $D_0$. Indeed, the Chu-Vandermonde identity gives us that 
\[\|(\frac{1-z}{2})^{m}\|^2  = 2^{-2m}\displaystyle \sum _{k=1}^m \binom {m}{k}^2 (k+1) 
= 2^{-2m}\displaystyle \sum _{k=1}^m \binom {m}{k}\binom {m-1}{m-k}  (k+1).\] An application of the Stirling approximation formula yields
\[\|(\frac{1-z}{2})^{m}\|^2 \eqsim  2^{-2m}  m\binom {2m-1}{m}
 = \frac{(2m-1)!}{2^{2m}((m-1)!)^2} \eqsim  \sqrt m .\]

Now, since $\|M_{\phi}^n\| \ge \|{\phi}^n\|$ and as far as $ \|(\frac{1-z}{2})^{kn}\|^2 \ge C \sqrt n$, $M_{\phi}$ cannot be power bounded.
\end{proof}

\bigskip

In terms of uniform mean ergodicity, we may deduce a  sufficient condition:

\begin{corollary} Let $\phi \in M_{D_0}$ and $\|\phi\|_{\infty}\le 1 $. If $\int_{\mathbb D} \log\frac{1}{1- |z|^2} |\phi'(z)|^2 dA <\infty $, and $1\notin \overline{ \phi(\mathbb D)}$, then $M_{\phi}$  is uniform mean ergodic (in the Dirichlet space).
\end{corollary}
\begin{proof}
    Under our hypotheses,  $\frac{\|T^n\|}{n} \rightarrow 0$ and $(I-T)D_0$ is closed. Then, by \cite{Li},   $T$ is uniform mean ergodic.
\end{proof}

\begin{remark}
Our example above, $\phi(z) = e^{i\theta} (\frac{1-z}{2})^{k}$ with $\theta \ne 0$ and $k\in \mathbb N$, $k \ge 1$, satisfies all the hypotheses of this Corollary. Thus $M_{\phi}$ is a uniform mean ergodic bounded operator in the Dirichlet space that fails to be power bounded.
\end{remark}

\subsection{Uniformly Kreiss boundedness}

We continue by characterizing  bounded multiplication operators in weighted Dirichlet spaces that are uniformly Kreiss bounded. As in the previous subsections, whenever we assume that a function is a multiplier in a particular space, any conclusion about its operator theoretic properties will refer to the action over the same particular space. Notice also that unimodular constant functions define trivially (UKB) multiplication operators. Larger constants are never (CB) while smaller ones are automatically (PB). Thus, from this point, we find it more convenient to exclude constants from our treatment.
 
\begin{theorem} Let $-1 < \alpha <1$ and $\phi \in M_{D_\alpha}$ be a non-constant function. Denote by $\mu_{n,\lambda}$ the measure with density given by \[d\mu_{n,\lambda}:= \left|\left(\frac{1-(\lambda\phi)^{n+1}} {(n+1)(1-\lambda \phi)}\right)'\right|^2 (1-|z|^2)^{\alpha} dA.\] Then $M_{\phi}$ and $M^*_{\phi}$ are (UKB) if and only if $\|\phi\|_{\infty}\le 1 $ and
$\{\mu_{n,\lambda}\}_{n\in\NN}$ is  a UBSCM for $D_{\alpha}$ for any $|\lambda | =1$.
\end{theorem}
\begin{proof}
It is sufficient to prove the characterization  for $M_{\phi}$, because being uniformly Kreiss bounded is preserved by the adjoint operation.

$(\Rightarrow) $ If the operator $M_{\phi}$ is (UKB) then the spectrum $\sigma (M_{\phi})$ is contained in $\overline{\mathbb D}$. Since $\sigma (M_{\phi})= \overline{\phi(\mathbb D)}$ then $\|\phi\|_{\infty} \le 1$.

In order to check that the sequence $\{\mu_{n,\lambda}\}$ is always actually a UBSCM for $D_{\alpha}$, we consider the embedding on an $f\in D_\alpha$. We want to bound the integral 
\[\frac{1}{\pi}\int_{\mathbb D} |f |^2 \left|\left(\frac{1-(\lambda\phi)^{n+1}}{(n+1)(1-\lambda \phi)}\right)'\right|^2 (1-|z|^2)^{\alpha}dA.\]
Notice that the derivative inside this integral can be replaced with that of $M_n (M_{\lambda\phi}1)$. Applying the product rule for the derivative, we can bound the integral above by 
\[\frac{2}{\pi}\int_{\mathbb D} |(M_n(\lambda M_{\phi} f))' |^2 (1-|z|^2)^{\alpha}dA 
+ \frac{2}{\pi}\int_{\mathbb D} |(M_n(\lambda M_{\phi}1) f' |^2 (1-|z|^2)^{\alpha}dA.\]
The first of these terms is bounded using the assumption that $M_\phi$ is (UKB). The other, using that $|\phi|\leq 1$.

$(\Leftarrow)$ In \cite[Corollary 3.2]{MSZ}, it is proved that an operator $T$ is  (UKB) if and only if there is a $C$ such that
$\|M_{n}(\lambda T)\| \le C \;\; \mbox{ for } |\lambda |=1  \mbox{ and } n=0,1,2, \dots$

Applied to our case, we want to bound the expression \[\|M_n(\lambda M_{\phi})f\|^2_{D_{\alpha}} =  \|M_n( M_{\lambda\phi})f\|^2_{D_{\alpha}}.\] As usual, the value at $0$ will not create any problems. For the remaining terms, the product rule for the derivative yields a bound from above again and this time it is given by
\[\frac{2}{\pi}\int_{\mathbb D} | M_n( M_{\lambda\phi}1) f' |^2(1-|z|^2)^{\alpha} dA + \frac{2}{\pi}\int_{\mathbb D} |(M_n M_{\lambda\phi}1)' f |^2 (1-|z|^2)^{\alpha}dA.\]
The fist of these integrals can be bounded by $\|f\|^2_{D_\alpha}$ using that $|\phi|\leq 1$. The other one corresponds exactly with the sequence of integrals that we can suppose bounded precisely because $\{\mu_{n, \lambda}$ are a UBSCM:

\[\frac{C}{\pi}\int_{\mathbb D} |f|^2|(\frac{1-(\lambda\phi)^{n+1}}{(n+1)(1-\lambda \phi)})'|^2(1-|z|^2)^{\alpha} dA\leq  C\|f\|^2_{D_{\alpha}}.\] This concludes the proof.
\end{proof}

When a few particular measures have a common Carleson bound, this will provide a simpler testing condition:

\begin{corollary} Let $-1 < \alpha <1$ and $\phi \in M_{D_\alpha}$.  Suppose that $\|\phi\|_{\infty}\le 1 $ and that the density $ |\frac{\phi'(z)}{1-\lambda\phi(z)}|^2(1-|z|^2)^{\alpha} dA $ defines a Carleson measure for $D_{\alpha}$ with a uniform embedding constant for all $|\lambda | =1$. Then $M_{\phi}$ and $M^*_{\phi}$ are (UKB).
\end{corollary}

\begin{proof}
Under these conditions, there is a $C$ such that
$\|M_{n}(\lambda M_{\phi})\| \le C \;\; \mbox{ for } |\lambda |=1  \mbox{ and } n=0,1,2, \dots$ Thus $ M_{\phi}$ is uniformly Kreiss.
\end{proof}

Whenever $\alpha \in [0,1)$ we will be able to give an even nicer testing condition. This follows from applying the reproducing kernel properties \eqref{kernel} as before:
\begin{corollary} Let $\phi \in M_{D_\alpha}$ with $0\le \alpha <1$ and suppose that $\|\phi\|_{\infty}\le 1 $.
\begin{enumerate}
\item[(a)] If $\alpha >0$ and $\int_{\mathbb D}|\frac{\phi'}{1-\lambda\phi}|^2 dA$ is uniformly bounded in $\lambda$ for $|\lambda|=1$, then $M_{\phi}$ and $M^*_{\phi}$ are (UKB).

\item[(b)] If $\alpha=0$ and $\int_{\mathbb D} \log\frac{1}{1- |z|^2} |\frac{\phi'(z)}{1- \lambda\phi(z)}|^2 dA $ is  uniformly bounded in $\lambda$ for $|\lambda|=1$, then $M_{\phi}$ and $M^*_{\phi}$ are (UKB).
\end{enumerate}
\end{corollary}

\begin{remark}
Recall that for an operator on a Hilbert space, (UKB) implies (ME) and thus, all of our criteria for uniform Kreiss boundedness also trivally imply mean ergodicity.
\end{remark}
Keep in mind the above remark while considering the following result, which was already mentioned as known in the previous Section but shows the potential applicability of our conditions:
\begin{corollary}
Let $\alpha >0$. Then $M^*_z$ is (UKB) in $D_\alpha$.
\end{corollary}

\begin{proof}

$\phi(z)=z$ satisfies  $\|\phi\|_{\infty}\le 1 $ and  using (\ref{inequality}), we can obtain that
$$
 \int_{\mathbb D}\frac{|f(z)|^2}{|1-\lambda z|^2}(1-|z|^2)^{\alpha} dA \le C\|f\|_{D_{\alpha}}, \; \; \forall \lambda \mbox{ with } |\lambda |=1.
 $$
Thus $\frac{(1-|z|^2)^{\alpha}}{|1-\lambda z|^2} dA $ is a Carleson measure for $D_{\alpha}$ with a uniform bound that works for all $|\lambda | =1$.
\end{proof}

\section{Further questions}\label{further}

We conclude now by presenting some problems that we leave open:

\bigskip

{\bf Problem 1:}
Does there exist a $\phi \in M_{D_0}$ defining a (UKB) operator $M_{\phi}$ that fails to be (PB) in the Dirichlet space?

\bigskip

In the present article, we have proved that if $\phi \in M_{D_\alpha}$ for some $\alpha \in (-1,1)$, then the mean ergodicity and Ces\`{a}ro boundedness of $M_{\phi}$ are equivalent. However, for an analogous result for $M^*_{\phi}$  we did need additional conditions. Thus, the most natural question is the following:
\bigskip

{\bf Problem 2:}
Let $\alpha\in (-1,1)$ and $\phi \in M_{D_\alpha}$. Is it true, then, that $M^*_{\phi}$ is (ME) if and only if $M_{\phi}$ is (CB)? Otherwise, for what functions $\phi$ may we have mean ergodicity for $M_{\phi}$ but not for  $M^*_{\phi}$?

\bigskip

One might follow other directions as well. A few promising ones include absolute Ces\`{a}ro boundedness, an extension to other families of spaces (including those $D_\alpha$ with $\alpha \leq -1$) or extending from multiplication and their adjoints to general Toeplitz operators.

 \bigskip
 
{ \small \noindent Antonio Bonilla, Departamento de An\'alisis Matem\'atico and Instituto de Matem\'aticas y Aplicaciones (IMAULL),  Universidad de La Laguna, C/Astrof\'{\i}sico Francisco S\'anchez, s/n, 38206 La Laguna, Tenerife, Spain\\
email: abonilla@ull.edu.es \\

\noindent Daniel Seco, Departamento de An\'alisis Matem\'atico and Instituto de Matem\'aticas y Aplicaciones (IMAULL),  Universidad de La Laguna, C/Astrof\'{\i}sico Francisco S\'anchez, s/n, 38206 La Laguna, Tenerife, Spain\\
email: dsecofor@ull.edu.es}

\end{document}